\documentclass[10pt]{amsart}
\usepackage{amscd}

\usepackage{amsmath,amssymb,enumerate, 
 epsfig,amsthm,sidecap,mathrsfs}
\usepackage{amsfonts} 
\usepackage{fancyhdr}
\usepackage{pb-diagram}
\usepackage{color}
\usepackage{hyperref}
\usepackage[mathscr]{euscript}
\usepackage{mathtools}

\let\oldmarginpar\marginpar
\renewcommand\marginpar[1]{\-\oldmarginpar[\raggedleft\footnotesize #1]%
{\raggedright\footnotesize #1}}

\newtheorem{theorem}{Theorem}[section]
\newtheorem{lemma}[theorem]{Lemma}
\newtheorem{corollary}[theorem]{Corollary}
\newtheorem{proposition}[theorem]{Proposition}

\theoremstyle{definition}
\newtheorem{definition}[theorem]{Definition}
\newtheorem{example}[theorem]{Example}

\theoremstyle{remark}
\newtheorem{remark}[theorem]{Remark}

\numberwithin{equation}{section}

\usepackage{paralist}

\def\Z{\mathbb Z}

\def\tnu{\widetilde{\nu}}
\def\F{\mathcal F}

\DeclareMathOperator{\Sym}{Sym}
\DeclareMathOperator{\SL}{SL}
\DeclareMathOperator{\Alt}{Alt}
\DeclareMathOperator{\Hom}{Hom}

\DeclareMathOperator{\Irr}{Irr}
\DeclareMathOperator{\tr}{tr}

\DeclareMathOperator{\Id}{Id}
\DeclareMathOperator{\Rep}{Rep}
\DeclareMathOperator{\Aut}{Aut}

\begin{document}
\title{Twisted Frobenius--Schur Indicators for Hopf Algebras}
\author{Daniel S.~Sage}\address{Department of Mathematics\\
  Louisiana State University\\
  Baton Rouge, LA 70803}\email{sage@math.lsu.edu}
\author{Maria D.~Vega}\email{vega@math.lsu.edu} \thanks{The research of the
  authors was partially supported by NSF grant~DMS-0606300 and
  NSA grant~H98230-09-1-0059.}  

\begin{abstract}

The classical Frobenius--Schur indicators for finite groups are
character sums defined for any representation and any integer $m\ge
2$. In the familiar case $m=2$, the Frobenius--Schur indicator
partitions the irreducible representations over the complex numbers
into real, complex, and quaternionic representations.  In recent
years, several generalizations of these invariants have been
introduced.  Bump and Ginzburg, building on earlier work of Mackey,
have defined versions of these indicators which are twisted by an
automorphism of the group. In another direction, Linchenko and
Montgomery have defined Frobenius--Schur indicators for semisimple
Hopf algebras.  In this paper, the authors construct twisted
Frobenius--Schur indicators for semisimple Hopf algebras; these include
all of the above indicators as special cases and have similar
properties.

\end{abstract}
\maketitle
\section{Introduction}

Classically, the Frobenius--Schur indicator of a character of a finite
group is the character evaluated at the sum of squares of the group
elements divided by the order of the group.  This indicator was
introduced by Frobenius and Schur in their investigation of real
representations.  Indeed, they showed that the only possible values
for an irreducible representation are $1$, $0$, and $-1$,
corresponding to the partition of the irreducible representations into
real, complex, and quaternionic representations~\cite{FrobSchur:1906}.
Higher order versions can be obtained by replacing squares with other
powers of group elements.

In recent years, there has been increasing interest in various
generalizations of these invariants.  In one direction, Bump and
Ginzburg~\cite{BumpGinz:2004}, building on earlier work of
Mackey~\cite{Mackey:1958} and Kawanaka-Matsuyama~\cite{KawaMats:1990},
have defined versions of Frobenius--Schur indicators which are twisted
by an automorphism of the group.  These indicators have applications
to the study of multiplicity-free permutation representations, models
for finite groups (in the sense of \cite{BGG:1976}), and Shintani
lifting of characters of finite reductive groups.

Another direction involves extending the theory from finite groups to
suitable Hopf algebras.  In 2000, Linchenko and Montgomery constructed
Frobenius--Schur indicators for semisimple Hopf algebras over an
algebraically closed field of characteristic zero and proved that the
second indicator again only takes the values $0$ or $\pm 1$ on
irreducible representations~\cite{LinMont:2000}.  The higher
indicators were further studied by Kashina, Sommerh\"auser, and Zhu,
who used them to prove a version of Cauchy's theorem for Hopf
algebras, namely that the dimension and exponent of a semisimple Hopf
algebra have the same prime factors~\cite{KashinaSommerZhu:2006}.  The
second indicators have also been used in classifying certain Hopf
algebras~\cite{Kashina:2003} and in studying possible dimensions of
representations~\cite{KashinaSommerZhu:2002}.  More recently, Ng and
Schauenberg have introduced a categorical definition of
Frobenius--Schur indicators for pivotal
categories~\cite{NgSchauenburg:2007b} and shown that the two
definitions coincide in the case of Hopf
algebras~\cite{NgSchauenburg:2008}.  This perspective has led to an
extension of Cauchy's theorem to spherical fusion
categories~\cite{NgSchauenburg:2007a}, applications to rational
conformal field
theory~\cite{Bantay:1997,Bantay:2000,NgSchauenburg:2007a}, and some
remarkable relations between certain generalizations of these
indicators and congruence subgroups of
$\SL_2(\Z)$~\cite{SommerZhu:2008,NgSchauenburg:2010}.

The goal of this paper is to construct twisted Frobenius--Schur
indicators for semisimple Hopf algebras over an algebraically closed
field of characteristic zero that include the group and Hopf algebra
indicators described above as special cases and have similar
properties.  Given an automorphism of order $n$ of such a Hopf
algebra, we define the \emph{$m$-th twisted Frobenius--Schur}
indicator for $m$ any positive multiple of $n$.  This definition is
given in Section~\ref{def}.  In the next section, we show that the
$m$-th twisted Frobenius--Schur indicator can be realized as the
trace of an endomorphism of order $m$ (Theorem~\ref{trace}), so that
the indicator is a cyclotomic integer.  In
Section~\ref{twistedsecond}, we consider the case of automorphisms of
order at most two.  We show that the second twisted Frobenius--Schur
indicator gives rise to a partition of the simple modules into three
classes; this partition involves the relationship between the module
and its ``twisted dual'' (Theorem~\ref{tfs2}).  In the final section,
we compute a closed formula for the twisted indicator of the regular
representation (Theorem~\ref{regular}).

\section{Definition} \label{def}
%


Let $k$ be an algebraically closed field of characteristic $0$, and
let $H$ be a semisimple Hopf algebra over $k$ with comultiplication
$\Delta$, counit $\varepsilon$, and antipode $S$.  The Hopf algebra
$H$ contains a unique two-sided integral $\Lambda$ normalized so that
$\varepsilon(\Lambda)=1$.  We will use the usual Sweedler notation for
iterated comultiplication:
$\Delta^{m-1}(\Lambda)=\sum_{(\Lambda)}{\Lambda_1 \otimes \Lambda_2
  \otimes \cdots \otimes \Lambda_m}$.  Let $\Rep(H)$ be the category
of finite-dimensional left $H$-modules.  All $H$-modules considered
will be objects in $\Rep(H)$.  Throughout the paper, an automorphism
of $H$ will always refer to a Hopf algebra automorphism (or
equivalently, a bialgebra automorphism).  In particular, such an
automorphism commutes with the antipode.

We are now ready to define the twisted indicators.  Let $\tau$ be an
automorphism of $H$ such that $\tau^m=\Id$ for some $m\in\mathbf N$.
Let $(V,\rho)$ be an $H$-module with corresponding character $\chi$.

\begin{definition} 
The \emph{$m$-th twisted Frobenius--Schur indicator of $(V,\rho)$} (or
$\chi$) is defined to be the character sum
\begin{equation}
\label{eq:tdmFS}
 \nu_m(\chi,\tau)= \sum_{(\Lambda)}{\chi\left(\Lambda_1 \tau\left(\Lambda_2\right)\cdots \tau^{m-1}\left(\Lambda_m\right)\right)}.
 \end{equation}
\end{definition}
We note that this is only defined for $m$ divisible by the order of
$\tau$.  We will write $\tnu_m(\chi)$ instead of $\nu_m(\chi,\tau)$ when
this does not cause confusion.

If $\tau=\Id$, this formula coincides with the definition of Linchenko
and Montgomery~\cite{LinMont:2000}.  Moreover, suppose $H=k[G]$ for a
finite group $G$.  In this case, $\Lambda=
\dfrac{1}{|G|}\sum_{g \in G}{g}$, and we recover Bump and Ginzburg's twisted
Frobenius--Schur indicators for groups~\cite{BumpGinz:2004}.

%

\section{A trace formula}\label{traceformula}

In this section, we realize $\tnu_m(\chi)$ as the trace of an
endomorphism of order $m$ and use this fact to show that the twisted
Frobenius--Schur indicators are cyclotomic integers.

 We begin by introducing a twisting functor
$\F_{\tau}: \Rep(H) \rightarrow \Rep(H)$.  Given an $H$-module $V$, we
set $\F_{\tau}(V)=V$ as vector spaces with $H$-action defined by
$h\cdot v= \tau(h)v$.  Furthermore, if $f:V \rightarrow W$ is a
morphism of $H$-modules, then $f: \F_{\tau}(V) \rightarrow
\F_{\tau}(W)$ is also an $H$-map; we set $\F_{\tau}(f)=f$.  The
functor $\F_{\tau}$ preserves the trivial module, tensor products, and
duals: 
\begin{equation}\label{rigid}
\F_{\tau}(k)=k, \F_{\tau}(V \otimes W)= \F_{\tau}(V) \otimes
\F_{\tau}(W),\text{ and }  \F_{\tau}(V^*)= \F_{\tau}(V)^*
\end{equation} 
for all $V,W \in \Rep(H)$.  In other words, $ \F_{\tau}$ is a strict,
rigid, $k$-linear endomorphism of $\Rep(H)$.  Moreover, if $\sigma$ is
another automorphism of $H$, then $ \F_{\sigma\tau}= \F_{\tau}
\F_{\sigma}$, so $\F_{\tau}$ is in fact an automorphism.

Let $\Aut_{\mathrm{sr}}(\Rep(H))$ denote the group of strict, rigid,
$k$-linear automorphisms of $\Rep(H)$.  Summing up, we obtain

\begin{proposition} The map $\tau\mapsto  \F_{\tau}$ is an
  anti-homomorphism $\Aut(H)\to\Aut_{\mathrm{sr}}(\Rep(H))$.  In particular, if
  $\tau^m=\Id$, then $\F_{\tau}^m=\Id$.
\end{proposition}

Let $(\widetilde{V^{\otimes m}},\widetilde{\rho^{m}})$ be the
$H$-module 
\begin{equation}
\label{tildeV}
 \widetilde{V^{\otimes m}}= V \otimes \F_{\tau}(V) \otimes \left( \F_{\tau}(V) \right) ^2\otimes \cdots \otimes \left( \F_{\tau}(V)\right) ^{m-1}.
\end{equation}
To be explicit, $\widetilde{V^{\otimes m}}$ has underlying vector space $V^{\otimes m}$ and action given by 
$$\widetilde{\rho^{m}}\left(h\right)\left(v_1 \otimes  v_2 \otimes \cdots \otimes v_m\right)= \sum_{(h)}{\rho\left(h_1\right)v_1 \otimes \rho\left(\tau\left(h_2\right)\right)v_2} \otimes \cdots \otimes \rho\left(\tau^{m-1}\left(h_m\right)\right)v_m.$$

Let $\alpha: V^{\otimes m} \rightarrow V^{\otimes m}$ be the linear
map defined by $$\alpha\left(v_1 \otimes v_2 \otimes \cdots \otimes v_m\right)= v_2 \otimes \cdots \otimes v_m \otimes v_1. $$

\begin{lemma}
\label{thtrace}
$$\tnu_{m}(\chi)=\tr_{V^{\otimes m}}\left(\alpha \circ
\widetilde{\rho^m}\left(\Lambda\right)\right).$$
\end{lemma}

\begin{proof}
\begin{eqnarray*} 
\tnu_{m}(\chi)
& = & \sum_{(\Lambda)}{\chi\left(\Lambda_1 \tau\left(\Lambda_2\right)\cdots \tau^{m-1}\left(\Lambda_m\right)\right)}\\
 & = & \sum_{(\Lambda)}{\tr_{V} \left(\rho\left(\Lambda_1\right)\rho\left(\tau\left(\Lambda_2\right)\right) \cdots \rho\left(\tau^{m-1}\left(\Lambda_m\right)\right)\right)}\\
 & = & \tr_{V^{\otimes m}}\left(\alpha \circ \left(\rho \otimes \rho\tau \otimes \cdots \otimes \rho\left(\tau^{m-1}\right)\right)\Lambda\right)\\
 & = & \tr_{V^{\otimes m}}\left(\alpha \circ \widetilde{\rho^m}\left(\Lambda\right)\right).
\end{eqnarray*}

The third equality uses \cite[Lemma 2.3]{KashinaSommerZhu:2002}.
\end{proof}

It is well known that the integral $\Lambda$ in $H$ is cocommutative, i.e.,

\begin{equation*}
\Delta\left(\Lambda\right)= \sum_{(\Lambda)}{\Lambda_1 \otimes \Lambda_2 }= \sum_{(\Lambda)}{\Lambda_2 \otimes \Lambda_1}.
\end{equation*}
More generally, 
$\Delta^m(\Lambda)$ is invariant under cyclic permutations:
\begin{equation}\label{eq:Lcycleinv}
\Delta^{m}\left(\Lambda\right)= \sum_{(\Lambda)}\Lambda_1 \otimes \Lambda_2 \otimes \cdots \otimes \Lambda_{m+1}= \sum_{(\Lambda)}\Lambda_2 \otimes \cdots \otimes \Lambda_m \otimes \Lambda_{m+1}\otimes \Lambda_1.
\end{equation}
Note that if $\sigma$ is an automorphism of $H$, then $\sigma(\Lambda)=\Lambda.$

\begin{lemma}\label{lem4.5}
$$ \sum_{(\Lambda)}\Lambda_1 \otimes \tau\left(\Lambda_2\right) \otimes \cdots \otimes \tau^{m-1}\left(\Lambda_{m}\right)= \sum_{(\Lambda)}\tau\left(\Lambda_2\right) \otimes \cdots \otimes \tau^{m-1}\left(\Lambda_m\right)\otimes \Lambda_1.$$

\end{lemma}
\begin{proof}
By the previous corollary, $\Delta^{m-1}\left(\Lambda\right)  =
\Delta^{m-1}\left(\tau^{m-1}(\Lambda)\right)$.  Since $\tau^{m-1}$ is a coalgebra
morphism, we get
\begin{align*}
\sum_{(\Lambda)}{\Lambda_1 \otimes \cdots \otimes \Lambda_m} & =  \sum_{(\Lambda)}{\tau^{-1}(\Lambda)_1 \otimes \cdots \otimes \tau^{-1}(\Lambda)_m}\\&= \sum_{(\Lambda)}{\tau^{-1}(\Lambda_1)  \otimes \cdots \otimes \tau^{-1}(\Lambda_m)}. 
\end{align*}





Combining this equation with (\ref{eq:Lcycleinv}) gives

\begin{eqnarray*}
\sum_{(\Lambda)}{\Lambda_2 \otimes \Lambda_3 \otimes \cdots \otimes \Lambda_m \otimes \Lambda_1} = \sum_{(\Lambda)}{\tau^{-1}(\Lambda_1) \otimes \tau^{-1}(\Lambda_2) \otimes \cdots \otimes \tau^{-1}(\Lambda_m)}.
\end{eqnarray*}

Applying $\left(\tau \otimes \tau^2 \otimes \cdots \otimes
\tau^m\right)$, we obtain

\begin{equation*}\sum_{(\Lambda)}{\tau\left(\Lambda_2\right) \otimes  \cdots \otimes
\tau^{m-1}\left(\Lambda_m\right) \otimes \Lambda_1} 
 =  \sum_{(\Lambda)}{\Lambda_1 \otimes \tau(\Lambda_2)\cdots \otimes \tau^{m-1}(\Lambda_m)},
\end{equation*}
as desired.
\end{proof}




It is well-known that the action of $\Lambda$ on an $H$-module $W$ gives a
projection onto its invariants.  Let $\pi: \widetilde{V^{\otimes m}}
\rightarrow \left(\widetilde{V^{\otimes m}}\right)^H$ defined by
$\pi(w)= \Lambda \cdot w$ be this projection for $W=\widetilde{V^{\otimes m}}$.
\begin{proposition}\label{pirestriction}
The linear automorphism  $\alpha$ restricts to an automorphism of $\left(\widetilde{V^{\otimes m}}\right)^H$.
\end{proposition}
\begin{proof}
  It is enough to show that $\left(\pi \circ \alpha \right)(w) =
  \left(\alpha \circ \pi\right)(w)$ for $w=v_1 \otimes \cdots \otimes
  v_m.$ Computing gives
\begin{eqnarray*}
\left(\pi \circ \alpha \right)(w) 
& = & (\pi\circ \alpha )(v_1 \otimes \cdots \otimes v_m)\\
& = & \pi \left(v_2 \otimes \cdots \otimes v_m \otimes v_1\right)\nonumber \\
  & = &\sum_{(\Lambda)}{\rho \left(\Lambda_1\right)v_2 \otimes \rho\left(\tau(\Lambda_2)\right)v_3 \otimes \cdots \otimes \rho\left(\tau^{m-1}(\Lambda_m)\right)v_1 },
\end{eqnarray*} 
and

\begin{eqnarray*}
\left(\alpha \circ \pi \right)(v) 
& = & \alpha \left(\Lambda \cdot (v_1 \otimes \cdots \otimes v_m)\right)  \\
 & = & \alpha \left(\sum_{(\Lambda)}{\rho(\Lambda_1))v_1 \otimes \rho(\tau(\Lambda_2))v_2 \otimes \cdots \otimes \rho(\tau^{m-1}(\Lambda_m))v_m }\right) \\
 & = &  \sum_{(\Lambda)}{\rho(\tau(\Lambda_2))v_2 \otimes \cdots
\otimes \rho(\tau^{m-1}(\Lambda_m))v_m  \otimes \rho(\Lambda_1)v_1)}.
\end{eqnarray*}
By Lemma~\ref{lem4.5}, these two expressions are equal.
\end{proof}

\begin{theorem}\label{trace} For any $V\in\Rep(H)$ with character
  $\chi$, the $m$-th twisted Frobenius--Schur indicator satisfies
$$\tnu_m\left(\chi\right)= \tr\left(\alpha|_{\left(\widetilde{V^{\otimes m}}\right)^H}\right).$$
\end{theorem}
\begin{proof} By Proposition~\ref{pirestriction}, the image of
$\alpha$ is contained in $\left(\widetilde{V^{\otimes m}}\right)^H$.
Moreover, its restriction to $\left(\widetilde{V^{\otimes m}}\right)^H$
coincides with the restriction of $\alpha$.  The result now follows by
Lemma~\ref{thtrace}.
\end{proof}

\begin{corollary}
Let $\zeta_m$ be a primitive $m$-th root of $1$, then 
$$\tnu_m\left(\chi\right) \in \Z\left[\zeta_m\right].$$  
\end{corollary}
\begin{proof}
The operator $\alpha$ is of order $m$, so its eigenvalues are $m$-th
roots of unity. It is now immediate from the theorem that the twisted
indicators are cyclotomic integers.
\end{proof}

As we will see below, when $m=2$, the twisted Frobenius--Schur
indicators are actually in $\Z$.

\section{Twisted second Frobenius--Schur indicators}\label{twistedsecond}

In this section, we will show that the second twisted Frobenius--Schur indicator
gives rise to a partition of the irreducible $H$-modules into three
classes, depending on the relationship between the module and its
\emph{twisted dual}.  We also compute the indicators for all
automorphisms of $H_8$--the smallest semisimple Hopf algebra that is
neither commutative nor cocommutative.

\subsection{Twisted duals and the partition of the simple modules}
Let $\tau$ be an automorphism such that $\tau^2=\Id$.  We will let
$T=\tau S$ denote the corresponding anti-involution.  Note that
$TS=ST$.  Let $(V,\rho)$ be a finite-dimensional $H$-module with
character $\chi$. Using \eqref{eq:tdmFS} for $m=2$, we have
$$\tnu_2(\chi)= \sum_{(\Lambda)}
\chi\left(\Lambda_1TS(\Lambda_2)\right).$$

\begin{definition}
 The \emph{twisted duality functor} $(-)^\dag:\Rep(H)\to\Rep(H)$ is the
 composition of $\F_{\tau}$ and the duality functor.
\end{definition}
In other words, $V^\dag$ is the dual space $V^*$ equipped with the
$H$-module structure given by $$(h \cdot f)(v)=f(T(h)\cdot v),$$ for
all $h\in H, f\in V^*$ and $v \in V$.  If $f:V\rightarrow W$ is an
$H$-map, then $f^\dag: W^{\dag}\rightarrow V^{\dag}$ is just the usual
dual map.

\begin{lemma}\label{doubletdual} There is an equality of functors
  $(-)^{\dag\dag}=(-)^{**}$.  In particular, $(-)^{\dag}$  is an
  involutory auto-equivalence of $\Rep(H)$.
\end{lemma}

\begin{proof} Equation \eqref{rigid} implies  
\begin{equation*}
V^{\dag \dag}= \F_{\tau}((\F_{\tau}(V^*))^{*})=\F_{\tau}^2(V^{**})=V^{**}
\end{equation*}
for any module $V$.  It is immediate that  $f^{\dag \dag}= f^{**}$ for
any $H$-map $f$.
\end{proof}

The lemma shows that the usual evaluation map $\Psi:V\to
V^{\dag\dag}$ given by $\Psi(v)(f)=f(v)$ is a canonical isomorphism of
$H$-modules.  We now define the transpose endomorphism on
$\Hom(V^\dag,V)$ via $f\mapsto\Psi^{-1} \circ f^\dag$.  (Since
$\Hom(V^\dag,V)$ and $\Hom(V^*,V)$ share the same underlying vector
space, this is just the usual transpose.)  In general, transposition
is not $H$-linear.  However, it is immediate that it restricts to give
an endomorphism of $ \Hom_H(V^\dag,V)$.  In particular, we can
consider symmetric and skew-symmetric $H$-maps $V^\dag\to V$:
$$\Sym_H(V^{\dag},V)=\left\lbrace f\in \Hom_H (V^{\dag},V)\vert f^t=f \right\rbrace $$
and 
$$\Alt_H(V^{\dag},V)=\left\lbrace f\in \Hom_H (V^{\dag},V)\vert f^t=-f\right\rbrace.$$




 
We can now state the main theorem of this section.
\begin{theorem}
\label{tfs2}
Let $V$ be an irreducible representation with character $\chi$.  Then the following properties hold: 
\begin{enumerate}
\item $\tnu_2(\chi)= 0,1,$ or $-1, \forall \chi \in \Irr(H)$.

\item $\widetilde{\nu_2}(\chi) \neq 0$ if and only if $V \cong V^{\dag}$. Moreover,  $\tnu_{2}(\chi)=1$ (resp. $-1$) if and only if there is a symmetric (resp. skew-symmetric) nonzero intertwining map $V \rightarrow V^{\dag}$.
\end{enumerate}
\end{theorem}

\begin{remark}
This result is well-known in two special cases.  If we let $T=S$
(i.e., $\tau=\Id$), then we recover Theorem 3.1 in
\cite{LinMont:2000}.  On the other hand, when $H$ is a group algebra,
this is a theorem of Sharp~\cite{Sharp:1960} and
Kawanaka-Matsuyama~\cite{KawaMats:1990}.  See also \cite{KwonSage:2008}.
\end{remark}

We will provide some preliminary results before proving the theorem.

There is a canonical $H$-isomorphism $Q:\widetilde{V^{\otimes 2}}\to
\Hom(V^{\dag},V)$ given by 
\begin{align*} \widetilde{V^{\otimes 2}}&=V \otimes \F_{\tau}(V)\\
&\cong V \otimes \F_{\tau}(V)^{**}\\
&\cong \Hom(\F_{\tau}(V)^*,V)\\
&=\Hom(V^{\dag},V).
\end{align*}
As a linear map, $Q$ is just the usual isomorphism $V\otimes V\to
\Hom(V^*,V)$ with $Q(v\otimes w)(\phi)=\phi(w)v$ for $v,w\in V$ and
$\phi\in V^*$. Thus, $Q\circ\alpha=(-)^t\circ Q$.  Taking
$H$-invariants and applying
Proposition~\ref{pirestriction}, we obtain the following lemma.

\begin{lemma}\label{conjugate}  There is a commutative diagram of $H$-maps 
\[
\begin{diagram}
  \node{\Hom_H(V^{\dag},V)} \arrow{e,t}{(-)^t}
  \node{\Hom_H(V^{\dag},V)} \\
  \node{(\widetilde{V^{\otimes 2}})^H}
  \arrow{n,l}{Q}\arrow{e,t}{\alpha} \node{(\widetilde{V^{\otimes
        2}})^H}\arrow{n,r}{Q}
\end{diagram}
\]
\end{lemma}

Let $\beta$ be the restriction of the transpose map to
$\Hom_H(V^{\dag},V)$.  The lemma says that $\beta$
is a conjugate of $\alpha|_{\left(\widetilde{V^{\otimes
        m}}\right)^H}$.  Since $\beta^2=\Id$, the eigenspace
decomposition of $\beta$ gives
\begin{equation}
   \label{eq:parity}\Hom_H(V^{\dag},V)=\Sym_H(V^{\dag},V) \oplus
   \Alt_H(V^{\dag},V).
 \end{equation}

\begin{proposition}\label{symalt3} Let $V$ be an $H$-module. Then,  
  \begin{equation*}
\tnu_2\left(\chi\right)= \dim \Sym_H(V^{\dag},V) - \dim \Alt_H(V^{\dag},V).
\end{equation*}
\end{proposition}
\begin{proof}  By \eqref{eq:parity}, the right side of this equation is
  $\tr(\beta)$. The assertion follows since
  $\tnu_2\left(\chi\right)=\tr(\beta)$ by Theorem~\ref{trace} and
  Lemma~\ref{conjugate}.
\end{proof}

\begin{remark}
The standard decomposition of $\Hom(V^{\dag},V)$ into symmetric and
skew-symmetric linear maps is not necessarily an $H$-decomposition.  In
fact, even when $\tau=\Id$, one need not get an $H$-decomposition unless $H$ is
cocommutative.
\end{remark}


\begin{proof}[Proof of Theorem~\ref{tfs2}]
Since $V$ is simple, it follows from Lemma~\ref{doubletdual} that
$V^\dag$ is also simple.  By Schur's Lemma, $\dim\Hom_H(V^\dag,V)\le
1$.  If $V^\dag\ncong V$, then $\dim\Hom_H(V^\dag,V)=0$, so
$\tnu_2\left(\chi\right)=0$ by Proposition~\ref{symalt3}.  Otherwise,
$V^\dag\cong V$, and Proposition~\ref{symalt3} shows that
$\tnu_2\left(\chi\right)$ equals $1$ or $-1$ depending on the parity
of any such isomorphism.
\end{proof}

\begin{remark}  One can also prove Theorem~\ref{tfs2} using the
  orthogonality relations for irreducible characters instead of
  Theorem~\ref{trace}.  Recall that  if the irreducible characters of $H$ are
given by $\chi_1,\dots,\chi_n$, then 
\begin{equation*}
\sum_{(\Lambda)}{ \chi_i\left(\Lambda_1\right)\chi_j\left(S\left(\Lambda_2\right)\right)}= \delta_{ij}.
\end{equation*}
(This is the dual statement of Theorem 7.5.6 in \cite{Dasc:2000}.)
Given a module $(V,\rho)$, the twisted dual
$(V^\dag,\widetilde{\rho})$ satisfies
$\widetilde{\rho}(h)=\rho(T(h))^t$.  Using this, one computes 
\begin{equation*} \tnu_2 \left( \chi \right)=\sum_{m,m'}{\sum_{(\Lambda)}{\rho(  \Lambda_1)_{mm'} \widetilde{\rho}( S(\Lambda _2 ))_{mm'}}}.
\end{equation*}
Now, assume that $V$ is simple.  If $V \not\cong V^{\dag}$, then this expression is $0$ by the
orthogonality relations.  Otherwise,  there exists
a nonzero intertwiner $\varphi\in \Hom_H(V^{\dag},V)$, so that
$\widetilde{\rho}(h)= \varphi^{-1}\rho(h)\varphi$; moreover, $\varphi$
is symmetric or skew-symmetric.  A calculation using the orthogonality
relations for matrix elements given in  \cite{Larson:1971} shows that
the above expression reduces to the parity of $\varphi$.
\end{remark}

\subsection{The second twisted Frobenius--Schur indicators for $H_8$}
\label{tFSH8section}

The smallest semisimple Hopf algebra which is neither commutative nor
cocommutative has dimension $8$.  We denote it by $H_8$.  As an
algebra, $H_8$ is generated by elements $x$, $y$ and $z$,  with relations: 
\begin{equation*}
x^2=y^2=1, \; z^2= \dfrac{1}{2}\left(1+x+y-xy\right), \;xy=yx,\; xz=zy, \text{ and } yz=zx.
\end{equation*}
The coalgebra structure of $H_8$ is given by the following:
$$\Delta(x)=x \otimes x,\; \varepsilon(x)=1, \text{ and } S(x)=x,$$
$$\Delta(y)=y \otimes y, \; \varepsilon(y)=1, \text{ and } S(y)=y,$$
$$\Delta(z)=\dfrac{1}{2}\left(1 \otimes 1 + 1 \otimes x + y \otimes 1
- y \otimes x\right)\left(z \otimes z\right),$$ $$ \varepsilon(z)=1,
\text{ and } S(z)=z.$$ The normalized integral is given by
$$\Lambda=\dfrac{1}{8}\left(1+x+y+xy+z+xz+yz+xyz\right).$$ This Hopf
algebra was first introduced by Kac and Paljutkin~\cite{KacPalj:1966}
and revisited later by Masuoka~\cite{Masuoka:1995}.

The Hopf algebra $H_8$ has $4$ one-dimensional representations and a
single two-dimensional simple module.  The characters for the
irreducible representations of $H_8$ are listed in
Table~\ref{tabH8irreps}.
\begin{center}
\begin{table}
\begin{tabular}{|c||c|c|c|c|c|c|c|c|}
\hline
 & $1$    & $x$ & $y$  & $xy$ & $z$ & $xz$ & $yz$ & $xyz$\\
\hline
\hline
$\chi_1$  & $1$  & $1$ & $1$  & $1$ & $1$ &1 & $1$& $1$\\
$\chi_2$  & $1$  & $1$ & $1$  & $1$ & $-1$ & $-1$& $-1$ & $-1$\\
$\chi_3$  & $1$  & $-1$ & $-1$  & $1$ & $i$ &$-i$ & $-i$ & $i$\\
$\chi_4$  & $1$  & $-1$ & $-1$  & $1$ & $-i$ &$i$ & $i$ &  $-i$ \\
$\chi_5$  & $2$  & $0$ & $0$  & $-2$ & $0$ & $0$ & $0$ & $0$\\
\hline
\end{tabular}
\vspace{3ex}
\caption{Characters for the Irreducible Representations of $H_8$}
\label{tabH8irreps}
 \end{table}
\end{center}

The automorphism group of $H_8$ is the Klein four-group.  These
automorphisms are given in Table~\ref{autotableforH8}.

\begin{center}
\begin{table}
\begin{tabular}{|c||c|c|c|c|}
\hline
& $1$ & $x$& $y$& $z$\\
\hline
\hline
$\tau_1=\Id$& $1$ & $x$& $y$& $z$\\
$\tau_2$& $1$ & $x$& $y$ & $xyz$\\
$\tau_3$& $1$ & $y$& $x$& $\frac{1}{2}\left(z+xz+yz-xyz\right)$\\
$\tau_4$& $1$ & $y$& $x$& $\frac{1}{2}\left(-z+xz+yz+xyz\right)$\\
\hline
\end{tabular}
\vspace{3ex}
\caption{Automorphisms of $H_8$}
\label{autotableforH8}
\end{table}
\end{center}

All four automorphisms satisfy $\tau^2=\Id$, so the second twisted Frobenius--Schur
indicator is defined for all of them.  These indicators are given in
Table~\ref{tindstable}.





\begin{center}
\begin{table}
\begin{tabular}{|c||c|c|c|c|c|}
\hline
&$\chi_1$&$\chi_2$&$\chi_3$&$\chi_4$&$\chi_5$\\
\hline
\hline
$\nu_2\left(\chi,\tau_1\right)=\nu_2(\chi)$&$1$&$1$&$1$&$1$&$1$\\
$\nu_2\left(\chi,\tau_2\right)$&$1$&$1$&$1$&$1$&$1$\\
$\nu_2\left(\chi,\tau_3\right)$&$1$&$1$&$0$&$0$&$1$\\
$\nu_2\left(\chi,\tau_4\right)$&$1$&$1$&$0$&$0$&$-1$\\
\hline
\end{tabular}
\vspace{3ex}
\caption{Twisted Frobenius--Schur indicators for $H_8$}
\label{tindstable}
\end{table}
\end{center}

\section{The regular representation}

We now return to the general case.  In this section, we realize the
twisted Frobenius--Schur indicators of the regular representation as
the trace of an explicit linear endomorphism of $H$.  Let $\chi_R$
denote the character of the left regular representation.

Let $\Omega^\tau_m: H \to H$ be the linear map defined by
$$\Omega^\tau_m(h)= \sum_{(h)}{S\left(\tau^{m-1}(h_1)\tau^{m-2}(h_2) \cdots \tau^2(h_{m-2})\tau(h_{m-1})\right)}.$$
\begin{theorem}\label{regular} The $m$-th twisted Frobenius--Schur
  indicator of the regular representation satisfies
 $$\tnu_m(\chi_R) = \tr(\Omega^\tau_m). $$
\end{theorem}

We will need two lemmas.

\begin{lemma}\label{cdlemma1} For any $h^1,\dots,h^{m-1}\in H$, 
\begin{equation*}
\begin{split}
  &\sum_{(\Lambda)}{\Lambda_1 h^1 \otimes \tau\left(\Lambda_2\right) h^2\otimes \cdots \otimes \tau^{m-2}(\Lambda_{m-1})h^{m-1}\otimes \tau^{m-1}(\Lambda_m)}\\
  &= \sum_{(\Lambda)}{\Lambda_1 \otimes \tau(\Lambda_2S(h^1_{m-1}))h^2
    \otimes \cdots \otimes \tau^{m-2}(\Lambda_{m-1}S(h^1_{2}))
    h^{m-1}\otimes \tau^{m-1}(\Lambda_mS(h^1_{1}))}.
\end{split}
\end{equation*}
\end{lemma}

\begin{proof} By \cite[Lemma 1.2(b)]{LarsonRadford:1988}, we have
\begin{equation*}\sum_{(\Lambda)}{\Lambda_1 h^1 \otimes \Lambda_2}= \sum_{(\Lambda)}{\Lambda_1 \otimes \Lambda_2S(h^1)}.
\end{equation*}
 Applying $\Id \otimes \Delta^{m-1}$ to both sides, we get
\begin{equation*}
\begin{split}
&\sum_{(\Lambda)}{\Lambda_1h^1 \otimes \Lambda_2 \otimes \cdots \otimes \Lambda_{m-1} \otimes \Lambda_m}\\
&=\sum_{(\Lambda)}{\Lambda_1 \otimes \Lambda_2S(h^1_{m-1})\otimes \Lambda_2S(h^1_{m-2})\otimes \cdots \otimes \Lambda_{m-1}S(h^1_2)\otimes \Lambda_mS(h^1_1)}.
\end{split}
\end{equation*}
We then apply $\Id \otimes \tau \otimes \tau^2 \otimes \cdots \otimes \tau^{m-1}$ to get
\begin{equation*}
\begin{split}
  &\sum_{(\Lambda)}{\Lambda_1h^1 \otimes \tau(\Lambda_2 )\otimes \cdots \tau^{m-2}(\Lambda_{m-1})\otimes \tau^{m-1}(\Lambda_m})\\
  &=\sum_{(\Lambda)}{\Lambda_1 \otimes
    \tau(\Lambda_2S(h^1_{m-1}))\otimes \cdots \otimes
    \tau^{m-2}(\Lambda_{m-1}S(h^1_2))\otimes \tau^{m-1}(\Lambda_m
    S(h^1_1))}.
\end{split}
\end{equation*}
The lemma follows by right multiplying this equation by $h^1 \otimes h^2 \otimes \cdots \otimes h^{m-1}\otimes 1$.
\end{proof}

Next, define a linear map $\psi:\widetilde{H}^{\otimes(m-1)}\rightarrow \widetilde{H}^{\otimes(m-1)}$ by
\begin{equation*}
\begin{split}
& \psi\left(h^1 \otimes h^2 \otimes \cdots \otimes h^{m-1}\right)\\&= \sum_{(h^1)}{\tau(S(h^1_{m-1}))h^2 \otimes \tau^2(S(h^1_{m-2}))h^3 \otimes \cdots \otimes \tau^{m-2}(S(h^1_2))h^{m-1}\otimes \tau^{m-1}(S(h^1_1))}.  
\end{split}
\end{equation*}

\begin{lemma}
$$\tr(\psi)=\tr\left(\alpha|_{\left(\widetilde{V^{\otimes
m}}\right)^H}\right).$$
\end{lemma}
\begin{proof} To prove the lemma, it suffices to find a linear
isomorphism $$\varphi:\widetilde{H}^{\otimes (m-1)}\to\left(H \otimes
\widetilde{H}^{\otimes (m-1)}\right)^H$$ making the diagram
\[
\begin{diagram}
\node{\widetilde{H}^{\otimes (m-1)}}\arrow{s,l}{\varphi} \arrow{e,t}{\psi}
\node{\widetilde{H}^{\otimes (m-1)}} \arrow{s,r}{\varphi}\\
\node{\left(H \otimes \widetilde{H}^{\otimes (m-1)}\right)^H} \arrow{e,b}{\alpha} \node{\left(H \otimes \widetilde{H}^{\otimes (m-1)}\right)^H}
\end{diagram}
\]
commute.
Recall that for any $H$-module $W$, there is a linear isomorphism
$W\to(H\otimes W)^H$ given by $w \mapsto \sum_{(\Lambda)}{\Lambda_1
\otimes \Lambda_2 w}$.  Let $\varphi$ be this isomorphism for $W= \widetilde{H}^{\otimes(m-1)}$.




Calculating gives
\begin{equation*}
\begin{split}
&\left(\alpha \circ \varphi\right)\left(h^1 \otimes h^2 \otimes \cdots \otimes h^{m-1}\right)\\
&= \sum_{(\Lambda)}{ \tau(\Lambda_2)h^1 \otimes \tau^2(\Lambda_3)h^2 \cdots \otimes \tau^{m-1}(\Lambda_m)h^{m-1}\otimes \Lambda_1}\\
&= \sum_{(\Lambda)}{\Lambda_1h^1 \otimes \tau(\Lambda_2)h^2 \cdots \otimes \tau^{m-2}(\Lambda_{m-1})h^{m-1}\otimes \tau^{m-1}(\Lambda_m)}\\
&=\sum_{(\Lambda)}{\Lambda_1 \otimes \tau(\Lambda_2S(h^1_{m-1}))h^2 \otimes  
 \cdots \otimes \tau^{m-2}(\Lambda_{m-1}S(h^1_{2})) h^{m-1}\otimes \tau^{m-1}(\Lambda_mS(h^1_{1}))}\\ 
&=\sum_{(\Lambda)}\Lambda_1 \otimes \tau(\Lambda_2)\tau(S(h^1_{m-1}))h^2 \otimes  
 \cdots \otimes\\
&\qquad\qquad\qquad\qquad\tau^{m-2}(\Lambda_{m-1})\tau^{m-2}(S(h^1_{2})) h^{m-1}\otimes \tau^{m-1}(\Lambda_m)\tau^{m-1}(S(h^1_{1}))\\
 &= \left(\varphi \circ \psi\right) \left(h^1 \otimes h^2 \otimes \cdots \otimes h^{m-1}\right).
\end{split}
\end{equation*}
Here, the second and third equalities use Lemmas~\ref{lem4.5} and
\ref{cdlemma1} respectively.
\end{proof}

\begin{proof}[Proof of Theorem~\ref{regular}]
By the previous lemma, we need only show that
$\tr(\psi)=\tr(\Omega^\tau_m).$ Choose a basis $b^1,
\cdots b^n \in H$ with dual basis $b_1^*, \cdots b_n^* \in
H^*$. Writing out $\tr(\psi)$ in terms of the induced basis on
$H^{\otimes m}$, we obtain

\begin{equation*}
\begin{split}
 \tr\left(\psi\right)&= \sum_{i_1, \cdots, i_{m-1}=1}^{n}{\left\langle
 b_{i_1}^* \otimes \cdots \otimes b_{i_{m-1}}^*,
 \psi\left(b^{i_1}\otimes \cdots \otimes b^{i_{m-1}}\right)
 \right\rangle}\\ 
& \begin{multlined}= \sum_{i_1, \cdots,
 i_{m-1}=1}^{n}{b_{i_1}^*(\tau(S(b^{i_1}_{m-1}))b^{i_2} )
 b_{i_2}^*(\tau^2(S(b^{i_1}_{m-2}))b^{i_3})}\cdots\\
 \qquad b_{i_{m-2}}^*(\tau^{m-2}(S(b^{i_1}_2))b^{i_{m-1}})b_{i_m-1}^*(\tau^{m-1}(S(b^{i_1}_1)))\end{multlined}\\
 &\begin{multlined} = \sum_{i_1, \cdots,
 i_{m-2}=1}^{n}{b_{i_1}^*(\tau(S(b^{i_1}_{m-1}))b^{i_2} )
 b_{i_2}^*(\tau^2(S(b^{i_1}_{m-2}))b^{i_3})} \cdots\\ 
 b_{i_{m-2}}^*\left(\tau^{m-2}(S(b^{i_1}_2))\tau^{m-1}(S(b^{i_1}_1))\right)\end{multlined}\\
 & =\cdots=
 \sum_{i_1=1}^{n}{b_{i_1}^*\left(\tau(S(b^{i_1}_{m-1}))
 \tau^2(S(b^{i_1}_{m-2})) \cdots
 \tau^{m-2}(S(b^{i_1}_2))\tau^{m-1}(S(b^{i_1}_1))\right)}\\ & =
 \sum_{i=1}^{n}{b_{i}^*\left(\tau(S(b^{i}_{m-1}))
 \tau^2(S(b^{i}_{m-2})) \cdots
 \tau^{m-2}(S(b^{i}_2))\tau^{m-1}(S(b^{i}_1))\right)}\\ & =
 \sum_{i=1}^{n}{b_{i}^*\left(S(\tau(b^{i}_{m-1}))
 S(\tau^2(b^{i}_{m-2})) \cdots
 S(\tau^{m-2}(b^{i}_2))S(\tau^{m-1}(b^{i}_1))\right)}\\ & =
 \sum_{i=1}^{n}{b_{i}^*\left(S(\tau^{m-1}(b^{i}_1)\tau^{m-2}(b^{i}_2))\cdots\tau^2(b^{i}_{m-2})\tau(b^{i}_{m-1})
 \right)}\\ & = \tr\left(\Omega^\tau_m\right),
\end{split}
\end{equation*}
as desired.
\end{proof}

\begin{example} 
We revisit the Hopf algebra $H_8$ described in Section
\ref{tFSH8section}.  The linear maps $\Omega^{\tau}_2$ from
Theorem~\ref{regular} are given in Table~\ref{OmegatableforH8}.
Computing the traces, one obtains the twisted Frobenius--Schur
indicators for the regular representation:  $\nu_2(\chi_R,\tau_1)=6$,
$\nu_2(\chi_R,\tau_2)=6$, $\nu_2(\chi_R,\tau_3)= 4$, and
$\nu_2(\chi_R,\tau_4)= 0$.  These can, of course, also be calculated
from the information in Table~\ref{tindstable}. 

\begin{center}
\begin{table}
\begin{tabular}{|c||c|c|c|c|}
\hline
& $\Omega^{\tau_1}_2$ & $\Omega^{\tau_2}_2$ & $\Omega^{\tau_3}_2$ &
$\Omega^{\tau_4}_2$\\
\hline
\hline
$1$  & $1$   & $1$  & $1$                                   & $1$\\
$x$  & $x$   & $x$  & $y$                                   & $y$\\
$y$  & $y$   & $y$  & $x$                                   & $x$\\
$xy$ & $xy$  & $xy$ & $xy$                                  & $xy$\\
$z$  & $z$   & $xyz$& $\frac{1}{2}\left(z+xz+yz-xyz\right)$ & $\frac{1}{2}\left(-z+xz+yz+xyz\right)$\\
$xz$ & $yz$  & $xz$ & $\frac{1}{2}\left(z+xz-yz+xyz\right)$ & $\frac{1}{2}\left(z-xz+yz+xyz\right)$\\
$yz$ & $xz$  & $yz$ & $\frac{1}{2}\left(z-xz+yz+xyz\right)$ & $\frac{1}{2}\left(z+xz-yz+xyz\right)$\\
$xyz$& $xyz$ & $z$  & $\frac{1}{2}\left(-z+xz+yz+xyz\right)$& $\frac{1}{2}\left(z+xz+yz-xyz\right)$\\
\hline
\end{tabular}
\vspace{3ex}
\caption{The linear maps $\Omega^{\tau}_2$ for $H_8$}
\label{OmegatableforH8}
\end{table}
\end{center}

\end{example}
\bibliographystyle{amsalpha}	
\bibliography{myrefs}		

\end{document}